\documentclass[11pt]{amsart}
\usepackage{amsthm}
\usepackage[english]{babel}
\usepackage[autostyle]{csquotes}
\usepackage{geometry}
\usepackage{hyperref}
\newtheorem{theorem}{Theorem}[section]
\newtheorem{lemma}[theorem]{Lemma}

\newtheorem{proposition}[theorem]{Proposition}
\newtheorem{corollary}[theorem]{Corollary}
\newtheorem{question}[theorem]{Question}

\theoremstyle{definition}
\newtheorem{definition}[theorem]{Definition}
\newtheorem{example}[theorem]{Example}

\theoremstyle{remark}
\newtheorem{remark}[theorem]{Remark}

\numberwithin{equation}{section}
\usepackage{amssymb}
\usepackage{amsmath}
\begin{document}
	
	\title[Misiurewicz Polynomials]{A note on Misiurewicz Polynomials}
	%    Remove any unused author tags.
	
	%    author one information
	\author[V.Goksel]{Vefa Goksel}
	\address{Mathematics Department\\ University of Wisconsin\\
		Madison\\
		WI 53706, USA}
	\email{goksel@math.wisc.edu}

	\subjclass[2000]{Primary }
	%    For articles to be published after 1 January 2010, you may use
	%    the following version:
	%\subjclass[2010]{Primary }
	
	
	\date{}
	
	\dedicatory{}
	
	\begin{abstract}
	Let $f_{c,d}(x)=x^d+c\in \mathbb{C}[x]$. The $c_0$ values for which $f_{c_0,d}$ has a strictly pre-periodic finite critical orbit are called Misiurewicz points. Any Misiurewicz point lies in $\bar{\mathbb{Q}}$. Suppose that the Misiurewicz points $c_0,c_1\in \bar{\mathbb{Q}}$ are such that the polynomials $f_{c_0,d}$ and $f_{c_1,d}$ have the same orbit type. One classical question is whether $c_0$ and $c_1$ need to be Galois conjugates or not. Recently there has been a partial progress on this question by several authors. In this note, we prove some new results when $d$ is a prime. All the results known so far were in the cases of period size at most $3$. In particular, our work is the first to say something provable in the cases of period size greater than $3$.
	\end{abstract}
	
	
	\date{}
	\maketitle
    \section{Introduction}
    
    Let $f(x)\in \mathbb{C}[x]$ be a polynomial of degree $d\geq 2$. We denote by $f^n(x)$ the $n$th iterate of $f(x)$ for $n\geq 1$. We also make the convention that $f^0(x)=x$. For a given $c\in \mathbb{C}$, the orbit of $c$ under $f$ is defined to be the set
    $$O_f(c) = \{f(c),f^2(c),\dots\}.$$

    The polynomial $f$ is called \textbf{post-critically finite} (\textbf{PCF}) if this orbit is finite for every critical point of $f$. Most polynomials are not post-critically finite, so such polynomials are rather special. In this paper, we will consider an even more special case, namely post-critically finite polynomials of the form $x^d+c\in \mathbb{C}[x]$, where $d\geq 2$. From now on, we set $f_{c,d}(x)=x^d+c$. Polynomials in this family are particularly nice, because they all have the unique critical point $0$.\\
  
    Now suppose $f_{c,d}$ is PCF, i.e. there exist integers $m,n\in \mathbb{Z}$ with $n\neq 0$ such that $f_{c,d}^m(0)=f_{c,d}^{m+n}(0)$. We say that $\boldsymbol{f_{c,d}}$ \textbf{has exact type} $\boldsymbol{(m,n)}$ if $n$ is the minimal positive integer such that $f_{c,d}^m(0)=f_{c,d}^{m+n}(0)$ and $f_{c,d}^k(0) \neq f_{c,d}^{k+n}(0)$ for any $k<m$. It is easy to see that if $m\neq 0$, then $m$ has to be at least $2$. A number $c_0$ for which $f_{c_0,d}$ has type $(m,n)$ with $m\geq 2$ is called a \textbf{Misiurewicz point of type} $\boldsymbol{(m,n)}$. Any Misiurewicz point of type $(m,n)$ is a root of a polynomial $G_{d,m,n}(c)\in \mathbb{Z}[c]$, which we call the \textbf{Misiurewicz polynomial of type} $\boldsymbol{(m,n)}$. So, in particular, all Misiurewicz points lie in $\overline{\mathbb{Q}}$.\\
    
    It is straightforward to check that for $c_0,c_1\in \mathbb{C}$, the polynomials $f_{c_0,d}$ and $f_{c_1,d}$ are affine conjugate to each other if and only if $c_0^{d-1}=c_1^{d-1}$. Milnor \cite{Milnor} asked the following question.
    \begin{question}
    Suppose that $f_{c_0,d}$ and $f_{c_1,d}$ have the same exact type $(m,n)$. Does it follow that $c_0^{d-1}$ and $c_1^{d-1}$ are Galois conjugates?
    \end{question}
In this note, we will study the following question, which is a more general version of Question $1.1$. It appears in (Question $9.8$, \cite{Benedetto}) in a different form. 
\begin{question}
	Suppose that $f_{c_0,d}$ and $f_{c_1,d}$ have the same exact type $(m,n)$. Does it follow that $c_0$ and $c_1$ are Galois conjugates?
\end{question}
    Before we talk about some recent partial progress on these questions, let us first precisely define the polynomial $G_{d,m,n}(c)$.
    \begin{definition}
    \cite{Silverman} We set $G_{d,0,n}(c) = \prod_{k|n} (f_{c,d}^k(0))^{\mu(\frac{n}{k})}$. For $m\neq 0$, we define $G_{d,m,n}(c)$ as follows: We first set
    $$F_{d,m,n}(c) = \prod_{k|n}\left(\frac{f_{c,d}^{m+k}(0)-f_{c,d}^m(0)}{f_{c,d}^{m-1+k}(0)-f_{c,d}^{m-1}(0)}\right)^{\mu(\frac{n}{k})}.$$
    Then, for $m\geq 2$, we define
   $$
   G_{d,m,n}(c) = \left\{
   \begin{array}{ll}
   F_{d,m,n}(c) & \text{if } n  \not|  \text{ }m-1 \\
   F_{d,m,n}(c)/F_{d,1,n}(c) & \text{if } n  \text{ }|  \text{ }m-1.
   \end{array}
   \right.
   $$
    \end{definition}
    See \cite{Hutz} for a proof that $G_{d,m,n}(c)$ is in fact a polynomial with integer coefficients.\\
    
    We also need to introduce the polynomials $H_{d,m,n}(c)\in\mathbb{Z}[c]$, which are the unique polynomials that satisfy $H_{d,0,1}(c)=1$, and $H_{d,m,n}(c^{d-1})=G_{d,m,n}(c)$ for $(m,n)\neq (0,1)$. The polynomials $H_{d,m,n}(c)$ arise when one works with the polynomials $g_{c,d}(x)=cx^d+1$ instead of $f_{c,d}(x)$ (see \cite{Buff} and \cite{BEK}). In other words, they can be defined by simply replacing $f_{c,d}$ with $g_{c,d}$ in Definition $1.3$. \\
    
    Question $1.1$ is equivalent to ask whether the polynomial $H_{d,m,n}(c)$ is irreducible over $\mathbb{Q}$ or not, and Question $1.2$ is equivalent to ask whether the polynomial $G_{d,m,n}(c)$ is irreducible over $\mathbb{Q}$ or not. From now on, whenever we say irreducible, we will mean irreducibility over $\mathbb{Q}$ (unless we state otherwise).\\
    
    We note that because of the relation given above, the irreducibility questions for the polynomials $G_{d,m,n}(c)$ and $H_{d,m,n}(c)$ are not equivalent when $d>2$, namely the irreducibility of $G_{d,m,n}(c)$ is a stronger condition than the irreducibility of $H_{d,m,n}(c)$.\\  
    
     We now summarize the known partial results regarding Question $1.1$ and Question $1.2$. Buff \cite{Buff} has shown that $H_{d,0,3}(c)$ is irreducible if and only if $d\not \equiv 1(\text{mod } 6)$. The author \cite{G} has proven that for any $m\geq 2$, $G_{d,m,1}(c)$ is irreducible when $d$ is a prime, and also that $G_{2,m,2}(c)$ is irreducible. Buff, Epstein and Koch \cite{BEK} have proven that for any $m\geq 2$, $H_{d,m,1}(c)$ and $H_{d,m,2}(c)$ have exactly $k$ irreducible factors when $d$ is a prime power, where $k$ is such that $d=p^k$ for some rational prime $p$. They have also proven that for any $m\geq 2$, $G_{2,m,3}(c)$ is irreducible, and $H_{8,m,3}(c)$ has exactly $3$ irreducible factors. These irreducibility results they have proven were corollaries of one of their main theorems (Theorem $19$, \cite{BEK}), which makes a somewhat surprising connection between the polynomial $H_{d,m,n}(c)\in \mathbb{Z}[c]$ and the reduced polynomial $\overline{H_{d,0,n}}(c)\in \mathbb{F}_p[c]$ when $d$ is a power of $p$. More precisely, it states that if the reduced polynomial $\overline{H_{d,m,n}}(c)\in \mathbb{F}_p[c]$ is irreducible over $\mathbb{F}_p$, then the polynomial $H_{d,m,n}(c)$ has exactly $k$ irreducible factors over $\mathbb{Q}$, where $d=p^k$ for some prime $p$. They also remark that the reduced polynomial $\overline{H_{d,0,n}}(c)\in \mathbb{F}_p[c]$ is irreducible only in the cases that show up in the above corollaries: $(d,n)=(p^k,1), (p^k,2), (2,3)$ or $(8,3)$.\\
     
      We now state our main result.
    \begin{theorem}
 	Let $d$ be a prime. Then, for any $m\geq 2$, the number of irreducible factors of $G_{d,m,n}(c)$ over $\mathbb{Q}$ is bounded from above by the number of irreducible factors of the reduced polynomial $\overline{G_{d,0,n}}(c)\in \mathbb{F}_d[c]$. In particular, if $\overline{G_{d,0,n}}(c)\in \mathbb{F}_d[c]$ is irreducible over $\mathbb{F}_d$, then $G_{d,m,n}(c)$ is irreducible over $\mathbb{Q}$. 
 \end{theorem}
 The following immediate corollary to this theorem recovers all the cases that the polynomial $G_{d,m,n}(c)$ is known to be irreducible. 
 \begin{corollary}
 	Let $d$ be a prime. Then, for any $m\geq 2$, $G_{d,m,1}(c)$, $G_{2,m,2}(c)$ and $G_{2,m,3}(c)$ are irreducible over $\mathbb{Q}$.
 \end{corollary}
\begin{proof}
Noting that each of $\overline{G_{d,0,1}}(c)=c \in \mathbb{F}_d[c],\text{ }\overline{G_{2,0,2}}(c)=c+1 \in \mathbb{F}_2[c]$, and $\overline{G_{2,0,3}}(c)=c^3+c+1 \in \mathbb{F}_2[c]$ is irreducible, the corollary follows from Theorem $1.4$.
\end{proof}
 We also obtain the following new irreducibility result.
\begin{corollary}
For any $m\geq 2$, $G_{3,m,2}(c)$ is irreducible over $\mathbb{Q}$.
\end{corollary}
\begin{proof}
We have $G_{3,0,2}(c)=c^2+1$, which is irreducible in $\mathbb{F}_3[c]$, hence the result again follows from Theorem $1.4$.
\end{proof}
\begin{remark}
Theorem $1.4$ does not imply the irreducibility of $G_{d,m,2}(c)$ for any prime $d>3$, because we have $G_{d,0,2}(c)=c^{d-1}+1$, and one can easily show that $c^{d-1}+1$ is always reducible in $\mathbb{F}_d[c]$ when $d>3$.  In fact, one can prove something much better; $c^{d-1}+1$ is reducible modulo every prime when $d>3$.
\end{remark}
Although Theorem $1.4$ does not prove any new irreducibility result when $n>3$, it provides an upper bound for the number of irreducible factors of the polynomial $G_{d,m,n}(c)$, which is independent of $m$. In particular, because of the way its proof proceeds, it reduces Question $1.2$ to perhaps a simpler problem. We illustrate this with the following example.
\begin{example}
	One of the simplest cases that $G_{d,m,n}(c)$ is not known to be irreducible is the case $d=2, n=4$. Since we have $\overline{G_{2,0,4}}(c)=(c^2+c+1)(c^4+c+1)\in \mathbb{F}_2[c]$, the proof of Theorem $1.4$ implies that if $G_{2,m,4}(c)$ is not irreducible for some $m\geq 2$, then there must exist polynomials $f(c),g(c)\in \mathbb{Z}[c]$ such that
	$$G_{2,m,4}(c)=[(c^2+c+1)^{M_{m,4}}+2f(c)][(c^4+c+1)^{M_{m,4}}+2g(c)],$$
	where $M_{m,4}=2^{m-1}$ if $m\not\equiv 1(\text{mod }4)$, and $M_{m,4}=2^{m-1}-1$ otherwise.
	MAGMA computations reveal that this does not happen for the small values of $m$ (thus $G_{2,m,4}(c)$ is irreducible), but whether this is the case for all $m$ or not remains open.
\end{example}
Using Theorem $1.4$ together with a result of Buff-Epstein-Koch \cite{BEK} and Dedekind's criterion \cite{Dedekind}, we also prove the following result about the number fields generated by Misiurewicz points.
\begin{theorem}
Let $d$ be a prime, and $c_0$ a root of $G_{d,m,n}(c)$. Set $K=\mathbb{Q}(c_0)$. Then we have $d\not | \text{ }[\mathcal{O}_K:\mathbb{Z}[c_0]]$.
\end{theorem}
Theorem $1.9$ has an arithmetic consequence for the critical orbit of $f_{c,d}$, see Corollary $3.5$ for details.\\

Finally, we introduce some notation that we will be using throughout the article. Let $K$ be a number field, and $\mathcal{O}_K$ its ring of integers. For any $a\in \mathcal{O}_K$, we denote by $(a)$ the ideal of $\mathcal{O}_K$ generated by $a$. We will also denote by $N_{K/\mathbb{Q}}(a)$ the norm of $a$ in the extension $K/\mathbb{Q}$. When the polynomial $f_{c_0,d}$ has type $(m,n)$, we will use the set $\{a_1,\dots,a_{m+n-1}\}$ to denote the critical orbit of $f_{c_0,d}$, where we set $a_i=f_{c_0,d}^i(0)$. Whenever we use $a_i$ for some $i>m+n-1$, we again obtain it by setting $a_i=f_{c_o,d}^{i}(0)$ and using the periodicity of $f_{c_0,d}$.
	\section{Proof of Theorem $1.4$}
	The goal in this section is to prove Theorem $1.4$. We first need to make some preparation. We start by recalling the main theorem of \cite{G}, as it will be crucial throughout the paper.
	\begin{theorem}
		\cite{G} Let $f_{c,d}(x)=x^d+c \in \bar{\mathbb{Q}}[x]$ be a PCF polynomial having exact type $(m,n)$ with $m\neq 0$. Set $K=\mathbb{Q}(c)$, and let $O_{f_{c,d}} = \{a_1, a_2,\dots , a_{m+n-1}\} \subset \mathcal{O}_K$ be the critical orbit of $f_{c,d}$. Then the following holds:\\
		
		\item[(a)] If $n \not|$ $i$, then $a_i$ is a unit.
		\item[(b)] If $d$ is a prime and $ n \text{ }| \text{ }i$, then one has $(a_i)^{M_{m,n}} = (d)$, where
		$$
		M_{m,n} = \left\{
		\begin{array}{ll}
		d^{m-1}(d-1) & \text{if } n  \not|  \text{ }m-1 \\
		(d^{m-1}-1)(d-1) & \text{if } n  \text{ }|  \text{ }m-1.
		\end{array}
		\right.
		$$
	\end{theorem}
	\begin{lemma}
	Let $p$ be a rational prime, and $c_0$ a root of $G_{d,m,n}(c)$, where $m\neq 0$. Set $K=\mathbb{Q}(c_0)$. Then we have $a_n = G_{d,0,n}(c_0)u$ for some unit $u$ in $\mathcal{O}_K$.
	\end{lemma}
\begin{proof}
We know from Theorem $2.1$ that $a_i$ is a unit in $\mathcal{O}_K$ for all $1\leq i\leq n-1$. It is also clear by the definition of $G_{d,0,n}(c)$ that $G_{d,0,n}(c_0)$ divides $a_n$ in $\mathbb{Z}[c_0]$. Note that the sequence $\{a_i\}_{i\geq1}$ is a rigid divisibility sequence (see \cite{Hamblen} for a definition of a rigid divisibility sequence and the proof of this fact), from which one sees that $G_{d,0,n}(c_0)$ is the primitive part of $a_n$ (Lemma $5.4$, \cite{Looper}). This implies that $\frac{a_n}{G_{d,0,n}(c_0)}$ divides $a_1\cdots a_{n-1}$ in $\mathbb{Z}[c_0]$. But, the product $a_1\cdots a_{n-1}$ is a unit in $\mathcal{O}_K$, hence $\frac{a_n}{G_{d,0,n}(c_0)}$ must be a unit in $\mathcal{O}_K$, which is what we wanted.
\end{proof}
The following lemma due to Buff-Epstein-Koch will also be crucial in the proof of Theorem $1.4$.
\begin{lemma}
\cite{BEK} Let $d$ be a rational prime, and define $M_{m,n}$ as in Theorem $2.1$. Then we have $G_{d,m,n}\equiv G_{d,0,n}^{M_{m,n}}(\text{mod }d)$ for all $n\geq 1$.
\end{lemma}
	\begin{lemma}
	Let $K$ be a number field, $p$ a rational prime, and $\alpha\in \mathcal{O}_K$. Then the ideal $(p,\alpha)$ is the unit ideal if and only if $N_{K/\mathbb{Q}}(\alpha)$ is relatively prime to $p$.
	\end{lemma}
    \begin{proof}
    Set $N=N_{K/\mathbb{Q}}(\alpha)$. If $N$ is relatively prime to $p$, then there exist $a,b\in \mathbb{Z}$ such that $aN+bp=1$, which will lie in the ideal, since clearly $N$ lies in the ideal. For the other direction, suppose that $(p,\alpha)$ is the unit ideal. Choose $a,b\in \mathcal{O}_K$ so that $ap+b\alpha=1$. Let $\sigma_1,\dots,\sigma_n$ be the embeddings of $K$. Recall that $N=\prod_{i=1}^{n}\sigma_i(\alpha)$. Then we have 
    $$N_{K/\mathbb{Q}}(ap+bk) = \prod_{i=1}^{n} (p\sigma_i(a)+\sigma_i(b)\sigma_i(\alpha))=1,$$
    which, after expanding, becomes
    $$pA+NB =1$$
    for some algebraic integers $A,B$, which clearly shows that $N$ has to be relatively prime to $p$, as desired.
    \end{proof}
\begin{lemma}
Let $K$ be a number field, and $p$ a rational prime. Choose $\alpha \in \mathcal{O}_K$ so that $K=\mathbb{Q}(\alpha)$. Let $f(x)\in\mathbb{Z}[x]$ be the minimal polynomial of $\alpha$. Suppose $\overline{f}(x)\in \mathbb{F}_p[x]$ factors as
$$\overline{f}(x)=g_1(x)^{e_1}\cdots g_k(x)^{e_k},$$
where $g_1(x),\dots,g_k(x)\in \mathbb{F}_p[x]$ are distinct and irreducible. Then, for any monic polynomial $h(x)\in \mathbb{Z}[x]$, $(p, h(\alpha))$ is not the unit ideal in $\mathcal{O}_K$ if and only if $g_i(x)|\overline{h}(x)$ in $\mathbb{F}_p[x]$ for some $1\leq i\leq k$.  
\end{lemma}
\begin{proof}
First suppose that $(p,h(\alpha))$ is not the unit ideal in $\mathcal{O}_K$. Then by Lemma $2.4$, $N_{K/\mathbb{Q}}(h(\alpha)) \equiv 0 (\text{mod }p)$. Recall that $N_{K/\mathbb{Q}}(h(\alpha)) = \text{Res}(f,h)$. Hence, we get $\text{Res}(f,h)\equiv 0 (\text{mod }p)$, which forces $\overline{f}$ and $\overline{h}$ to have a common factor in $\mathbb{F}_p[x]$, which proves this part of the statement. For the other direction, assume $g_i(x)|\overline{h}(x)$ for some $1\leq i\leq k$. Since this means that $\overline{f}$ and $\overline{h}$ have a common factor in $\mathbb{F}_p[x]$, this again implies that $N_{K/\mathbb{Q}}(h(\alpha))=\text{Res}(f,h)\equiv 0(\text{mod }p)$, which, by Lemma $2.4$, shows that $(p,h(\alpha))$ is not the unit ideal in $\mathcal{O}_K$, as desired.
\end{proof}
The next proposition combined with the remark following it will provide us an explicit factorization of the ideal $(d)$ in the number field generated by a root of the Misiurewicz polynomial $G_{d,m,n}(c)$, which will be heavily used in the proof of Theorem $1.4$.
\begin{proposition}
Let $d$ be a prime. Suppose $\overline{G_{d,0,n}}(c)$ factors as
$$\overline{G_{d,0,n}}(c) = f_1(c)\cdots f_k(c),$$
where $f_1(c),\dots,f_k(c)\in \mathbb{F}_p[c]$ are distinct irreducible polynomials. Then, if $\tilde{f_1}(c),\dots,\tilde{f_k}(c)\in \mathbb{Z}[c]$ are any lifts of these polynomials, and $c_0$ is a root of $G_{d,m,n}$, we have
\begin{equation}
(a_n) = (d,\tilde{f_1}(c_0))\cdots (d,\tilde{f_k}(c_0)).
\end{equation}
\end{proposition}
\begin{proof}
First note that from Lemma $2.2$, we have $a_n=G_{d,0,n}(c_0)u$ for some unit $u\in \mathbb{Z}[c_0]$. This gives that
\begin{equation}
a_n=\tilde{f_1}(c_0)\cdots \tilde{f_k}(c_0)u+d\alpha(c_0)
\end{equation}  for some $\alpha(c)\in \mathbb{Z}[c]$. We will now prove the proposition by showing that each side of (2.1) is contained in the other side:\\

$\supseteq:$ All the generators of the product ideal involving $d$ already belong to $(a_n)$, because from Theorem $2.1$ we have $d\in (a_n)$. So, it suffices to show that $\tilde{f_1}(c_0)\cdots \tilde{f_k}(c_0)\in (a_n)$. We have $d\in (a_n)$, which gives $a_n-d\alpha(c_0) = \tilde{f_1}(c_0)\cdots \tilde{f_k}(c_0)u \in (a_n)$, which gives what we want, since $u$ is a unit.\\

$\subseteq:$ If $k\leq M_{m,n}$, then since $a_n\in (d,\tilde{f_i}(c_0))$ for all $i$ (by (2.2)), we get that $d$ lies in the right-hand side of (2.1), because from Theorem $2.1$ we have $d\in (a_n)^{k}$, and $a_n^{k}$ lies in the right-hand side. But then, if $d$ lies in the right-hand side of (2.1), we get that $a_n=\tilde{f_1}(c_0)\cdots \tilde{f_k}(c_0)u+d\alpha(c_0)$ lies in the right-hand side of (2.1) as well, as desired.  So, we can assume without loss of generality that $k>M_{m,n}$. By the reasoning above, to finish the proof, it suffices to prove that $d$ lies in the right-hand side of (2.1). Write $k=M_{m,n}l+q$, $0\leq q < M_{m,n}$. Note that similar to above, we will have $$a_n\in (d,\tilde{f}_{iM_{m,n}+1}(c_0))\cdots (d,\tilde{f}_{(i+1)M_{m,n}}(c_0))$$
for $i=0,\dots,l-1$, and
$$a_n\in (d,\tilde{f}_{lM_{m,n}+1}(c_0))\cdots (d,\tilde{f}_{lM_{m,n}+q}(c_0)).$$
This implies that $a_n^{l+1}$ lies in the right-hand side of (2.1), which, if $l+1\leq M_{m,n}$, will again imply that $d$ lies in the right-hand side of (2.1), which will finish the proof. If $l+1>M_{m,n}$, we can repeat the same argument again, and it is obvious that this procedure will eventually terminate, and we will get that $d$ lies in the right-hand side of (2.1), so we are done.
\end{proof}
\begin{remark}
Note that since we have $(a_n)^{M_{m,n}} = (d)$ from Theorem $2.1$, Proposition $2.3$ gives a factorization of the ideal $(d)$ in $\mathcal{O}_K$. More precisely, we get
\begin{equation}
(d) = (d,\tilde{f_1}(c_0))^{M_{m,n}}\cdots (d,\tilde{f_k}(c_0))^{M_{m,n}}.
\end{equation}
\end{remark}
We are finally ready to prove Theorem $1.4$.
\begin{proof}[Proof of Theorem $1.4$]
Recall from Lemma $2.3$ that if $\overline{G_{d,0,n}}(c)\in \mathbb{F}_d[c]$ factors as
$$\overline{G_{d,0,n}}(c) = f_1(c)\cdots f_k(c),$$
then we have
$$\overline{G_{d,m,n}}(c) = [f_1(c)\cdots f_k(c)]^{M_{m,n}}.$$
Let $H(c)\in \mathbb{Z}[c]$ be any irreducible factor of $G_{d,m,n}(c)$, and take $c_0$ to be a root of $H(c)$. If we can show that $\overline{H}(c) = [A(c)]^{M_{m,n}}$ for some $A(c)\in \mathbb{F}_d[c]$, this will clearly prove the theorem. Assume for the sake of contradiction that $\overline{H}(c) = f_1(c)^{\alpha_1}\cdots f_k(c)^{\alpha_k}$, where for at least one $i$ we have $0<\alpha_i<M_{m,n}$. This gives $H(c) = f_1(c)^{\alpha_1}\cdots f_k(c)^{\alpha_k}+dH_1(c)$ for some $H_1(c)\in \mathbb{Z}[c]$. In particular, we have $f_1(c_0)^{\alpha_1}\cdots f_k(c_0)^{\alpha_k} = -dH_1(c_0)$. The last equality implies that the product $(d,\tilde{f}_1(c_0))^{\alpha_1}\cdots (d,\tilde{f}_k(c_0))^{\alpha_k}$ is contained in the ideal $(d)$, because all the generators of the product ideal are divisible by $d$. This gives
\begin{equation}
(d) | (d,\tilde{f}_1(c_0))^{\alpha_1}\cdots (d,\tilde{f}_k(c_0))^{\alpha_k}.
\end{equation}
Now (2.3) and (2.4) together will clearly imply that if $\alpha_i<M_{m,n}$, then $(d,\tilde{f}_i(c_0))$ must be the unit ideal in $\mathcal{O}_K$, which contradicts Lemma $2.5$. Hence, we conclude that for all $i$ we have $\alpha_i=0$ or $\alpha_i=M_{m,n}$, which shows that $\overline{H}(c) = [A(c)]^{M_{m,n}}$ for some $A(c)\in \mathbb{F}_d[c]$, as desired. 
\end{proof}
\section{Proof of Theorem $1.9$}
The goal of this section is to prove Theorem $1.9$. We start by recalling a basic fact from algebraic number theory:
\begin{theorem}
\cite{Lang} Let $K/\mathbb{Q}$ be an algebraic number field of degree $n$ with ring of integers $\mathcal{O}_K$ and discriminant $D_K$. Let $\alpha\in \mathcal{O}_K$ with minimal polynomial $f(x)$ be such that $K=\mathbb{Q}(\alpha)$. Then
$$\text{Disc}(f(x)) = [\mathcal{O}_K:\mathbb{Z}[\alpha]]^2D_K.$$
\end{theorem}
Understanding the rational primes which divide the index $[\mathcal{O}_K:\mathbb{Z}[\alpha]]$ is important for the following reason: By Dedekind's Factorization Theorem, for a rational prime $p$ not dividing the index $[\mathcal{O}_K:\mathbb{Z}[\alpha]]$, the factorization of the ideal $(p)$ in $\mathcal{O}_K$ can be obtained from the factorization of the reduced polynomial $\overline{f}(x)\in \mathbb{F}_p[x]$ (See for instance \cite{Marcus} for a precise statement).\\

Next, we recall Dedekind's criterion, which will be the most important tool for the proof of Theorem $1.9$.
\begin{theorem}\cite{Dedekind}
Let $\alpha$ be an algebraic integer, $f$ its minimal polynomial, $K=\mathbb{Q}(\alpha)$, and $\mathcal{O}_K$ its ring of integers. Let $p$ be a rational prime. Let $\overline{f} = f_1^{e_1}\cdots f_k^{e_k}$ be the decomposition of $\overline{f}$ in $\mathbb{F}_p[x]$. Let $\tilde{f}_i\in \mathbb{Z}[x]$ be any lift of $f_i$, and $g\in \mathbb{Z}[x]$ such that $f=f_1^{e_1}\cdots f_k^{e_k}+pg$. The following are equivalent:\\
\item[(i)] $p\not | \text{ }[\mathcal{O}_K:\mathbb{Z}[\alpha]]$.\\
\item[(ii)] For all $i$, either $e_i=1$ or $f_i$ does not divide $\overline{g}$ in $\mathbb{F}_p[x]$.
\end{theorem}
We also need the following lemma, which is a special case of (Lemma $23$, \cite{BEK}). We give an alternative proof in this special case.
\begin{lemma}
Let $d$ be a prime. Then we have$$
\text{Res}(G_{d,m,n},G_{d,0,k}) = \left\{
\begin{array}{ll}
\pm d^{\text{deg}(G_{d,0,n})} & \text{if } n=k \\
\pm 1 & \text{if } n\neq k.
\end{array}
\right.
$$
\end{lemma}
\begin{proof}
Suppose $G_{d,m,n}(c)\in \mathbb{Z}[c]$ factors as
$$G_{d,m,n}(c)=H_1(c)\cdots H_l(c)$$
for some $H_1(c),\dots,H_l(c)\in \mathbb{Z}[c]$, and let $c_1,\dots,c_l$ be some roots of $H_1(c),\dots,H_l(c)$, respectively. Set $K_i=\mathbb{Q}(c_i)$ for $i=1,\dots,l$. Also define $a_{s}^{(i)}=f_{c_i,d}^s(0)$ for $i=1,\dots,l$. Note that we have 
$$\text{Res}(G_{d,m,n},G_{d,0,k})=\prod_{i=1}^{l}\text{Res}(H_i,G_{d,0,k}).$$
Recall as in the proof of Lemma $2.5$ that $\text{Res}(H_i,G_{d,0,k})=N_{K_i/\mathbb{Q}}(G_{d,0,k}(c_i))$ for $i=1,\dots,l$. Then, if $k=n$, we obtain
$$\text{Res}(H_i,G_{d,0,k})=\text{Res}(H_i,G_{d,0,n})=N_{K_i/\mathbb{Q}}(G_{d,0,n}(c_i))=\pm N_{K_i/\mathbb{Q}}(a_n^{(i)})=\pm d^{\frac{\text{deg}(H_i)}{M_{m,n}}},$$
where the third equality follows from Lemma $2.2$, and the last equality follows by using the fact that norm is multiplicative, because we have $(a_n^{(i)})^{M_{m,n}}=(d)$ in $\mathcal{O}_{K_i}$ (from Theorem $2.1$). Thus, we get
$$\text{Res}(G_{d,m,n},G_{d,0,k})=\prod_{i=1}^{l}\text{Res}(H_i,G_{d,0,k})=\pm d^{\frac{\text{deg}(G_{d,m,n})}{M_{m,n}}}=\pm d^{\text{deg}(G_{d,0,n})},$$
which gives us the result we want. Note that we used Lemma $2.3$ for the last equality. Now assume $k\neq n$. First note that we will be done if we can show that $\text{Res}(H_i,G_{d,0,k})=\pm 1$ for $i=1,\dots,l$. There are two cases: Either $n|k$ or $n\not|\text{ }k$. If $n|k$, since the sequence $\{a_j^{(i)}\}_{j\geq 1}$ is a rigid divisibility sequence, and $G_{d,0,k}(c_i)$ is the primitive part of $a_k^{(i)}$ for $i=1,\dots,l$, it follows that $G_{d,0,k}(c_i)$ divides $\frac{a_k^{(i)}}{a_n^{(i)}}$ in $\mathbb{Z}[c_i]$, which, by Theorem $2.1$, implies that $G_{d,0,k}(c_i)$ is a unit in $\mathcal{O}_{K_i}$, i.e. $N_{K_i/\mathbb{Q}}(G_{d,0,k}(c_i))=\pm 1$, which gives
$\text{Res}(H_i,G_{d,0,k})=\pm 1$. If $n\not|\text{ }k$, then $a_k^{(i)}$ is a unit in $\mathcal{O}_{K_i}$ for $i=1,\dots,l$, but $G_{d,0,k}(c_i)$ divides $a_k^{(i)}$ in $\mathbb{Z}[c_i]$, hence $G_{d,0,k}(c_i)$ is a unit in $\mathcal{O}_{K_i}$, i.e. $N_{K_i/\mathbb{Q}}(G_{d,0,k}(c_i))=\pm 1$, which again implies that $\text{Res}(H_i,G_{d,0,k})=\pm 1$, as desired.
\end{proof}
We are finally ready to prove Theorem $1.9$.
\begin{proof}[Proof of Theorem $1.9$]
Using the proof of Theorem $1.4$, we can write
the factorization of $G_{d,m,n}(c)$ over $\mathbb{Q}$ as $$G_{d,m,n}(c) = [A_1(c)^{M_{m,n}}+dB_1(c)]\cdots [A_l(c)^{M_{m,n}}+dB_l(c)]$$
for some $A_1,\dots,A_l,B_1,\dots,B_l\in \mathbb{Z}[c]$, and note that $A_i(c)\in \mathbb{Z}[c]$ are not necessarily irreducible. By Dedekind's criterion, to prove that $d\not | \text{ }[\mathcal{O}_K:\mathbb{Z}[c_0]]$ for any root $c_0$ of $G_{d,m,n}(c)$, it suffices to show that $\overline{A}_i(c)$ and $\overline{B}_i(c)$ have no common factor in $\mathbb{F}_d[c]$ for $i=1,\dots,l$. To prove this, we will do some computations with resultants.\\

Using Lemma $2.3$, we can write
\begin{equation}
G_{d,0,n}(c) = A_1(c)\cdots A_l(c)+dG(c)
\end{equation}
for some $G(c)\in \mathbb{Z}[c]$. First let
\begin{equation}
X_1=\text{Res}(A_1(c)^{M_{m,n}}+dB_1(c),G_{d,0,n}(c))\text{Res}(A_2(c)\cdots A_l(c),G_{d,0,n}(c)).
\end{equation}
Hence, we have
\begin{equation}
X_1= \text{Res}(A_1(c)^{M_{m,n}}A_2(c)\cdots A_l(c)+dA_2(c)\cdots A_l(c)B_1(c),G_{d,0,n}(c))
\end{equation}
Using (3.1), this gives
\begin{equation}
X_1=\text{Res}(G_{d,0,n}(c)A_1(c)^{M_{m,n}-1}+dA_2(c)\cdots A_l(c)B_1(c)-dG(c)A_1(c)^{M_{m,n}-1},G_{d,0,n}(c))
\end{equation}
Thus, by the basic properties of resultants, we get
\begin{align}
X_1  &=  \text{Res}(d(A_2(c)\cdots A_l(c)B_1(c)-G(c)A_1(c)^{M_{m,n}-1}),G_{d,0,n}(c))\\
 &=  d^{\text{deg}(G_{d,0,n})}\text{Res}(A_2(c)\cdots A_l(c)B_1(c)-G(c)A_1(c)^{M_{m,n}-1},G_{d,0,n}(c)).
\end{align}
On the other hand, using (3.1) in the second factor in (3.2), we also have
\begin{align}
X_1 &= \text{Res}(A_1(c)^{M_{m,n}}+dB_1(c),G_{d,0,n}(c))\text{Res}(A_2(c)\cdots A_l(c),A_1(c)\cdots A_l(c)+dG(c))\\
&=\text{Res}(A_1(c)^{M_{m,n}}+dB_1(c),G_{d,0,n}(c))\text{Res}(A_2(c)\cdots A_l(c), G(c))d^{\text{deg}(A_2(c)\cdots A_l(c))}
\end{align}
Hence, in (3.6) and (3.8), we obtained two different expressions for $X_1$. Doing the same thing for each $1\leq i\leq l$, and multiplying out $X_i$s, we will obtain two different expressions for the product $X_1\cdots X_l$. Namely, if we write each $X_i$ similarly to (3.6), we get
\begin{equation}
X_1\cdots X_l = d^{l\text{deg}(G_{d,0,n})}\prod_{i=1}^{l}\text{Res}(\frac{A_1(c)\cdots A_l(c)}{A_i(c)}B_i(c)-G(c)A_i(c)^{M_{m,n}-1},G_{d,0,n}).
\end{equation}
On the other hand, if we write each $X_i$ similarly to (3.8), we obtain
\begin{align}
X_1\dots X_l  &=  \text{Res}(\prod_{i=1}^{l}(A_i(c)^{M_{m,n}}+dB_i(c)),G_{d,0,n}(c))\text{Res}(A_1(c)\cdots A_l(c), G(c))^{l-1}d^{(l-1)\text{deg}(G_{d,0,n})}\\
&=\text{Res}(G_{d,m,n}(c),G_{d,0,n}(c))\text{Res}(A_1(c)\cdots A_l(c), G(c))^{l-1}d^{(l-1)\text{deg}(G_{d,0,n})}\\
&=  \pm d^{l\text{deg}(G_{d,0,n})}\text{Res}(A_1(c)\cdots A_l(c), G(c))^{l-1},
\end{align} where the last equality follows from Lemma $3.3$.
Hence, equating (3.9) and (3.12), and simplifying, we get
\begin{equation}
\pm[\text{Res}(A_1(c)\cdots A_l(c), G(c))]^{l-1} = \prod_{i=1}^{l}\text{Res}(\frac{A_1(c)\cdots A_l(c)}{A_i(c)}B_i(c)-G(c)A_i(c)^{M_{m,n}-1},G_{d,0,n}).
\end{equation}
Recall that our goal was to show that $\overline{A}_i(c)$ and $\overline{B}_i(c)$ have no common factors in $\mathbb{F}_d[c]$. Recalling (3.1), it is clear that to prove this, it suffices to show that the right-hand side of (3.13) is not divisible by $d$. So, we will be done if we can show that $\text{Res}(A_1(c)\cdots A_l(c), G(c))$ is not divisible by $d$, i.e., it is enough to show that $\text{Res}(A_i(c),G(c))$ is not divisible by $d$ for each $i$. We need the following lemma to achieve this.
\begin{lemma}
	Let $p$ be a rational prime, $f(x)\in \mathbb{Z}[x]$ a monic polynomial (not necessarily irreducible) such that $\text{Disc}(f)$ is relatively prime to $p$. Suppose that the reduced polynomial $\overline{f}(x)\in \mathbb{F}_p[x]$ factors as
	$$\overline{f}(x) = f_1(x)\cdots f_k(x),$$
	where $f_1(x),\cdots,f_k(x)\in \mathbb{F}_p[x]$ are irreducible. Then we can choose lifts $\tilde{f_1}(x),\cdots,\tilde{f_k}(x)\in \mathbb{Z}[x]$ of $f_1(x),\dots,f_k(x)$, respectively, and write $f(x)=\tilde{f_1}(x)\cdots \tilde{f_k}(x)+pF(x)$ such that $\text{Res}(\tilde{f_i}(x),F(x))$ is relatively prime to $p$ for all $i$.
\end{lemma}
\begin{proof}[Proof of Lemma $3.4$]
	Considering the factorization of $f(x)\in \mathbb{Z}[x]$, without loss of generality, we can write
	$$f(x)=[\tilde{f}_1(x)\cdots \tilde{f}_{i_1}(x)+pF_1(x)]\cdots [\tilde{f}_{i_{l-1}+1}(x)\cdots \tilde{f}_{i_l}(x)+pF_l(x)].$$
	Then we have $f(x)=\tilde{f}_1(x)\cdots \tilde{f}_k(x)+pF(x)$, where $$F(x)=F_1(x)\tilde{f}_{i_{1}+1}(x)\cdots \tilde{f}_{i_l}(x)+\tilde{f}_1(x)K(x)+pL(x)$$
	for some $K(x), L(x)\in \mathbb{Z}[x]$. We would like to have that $f_1(x)$ and $\overline{F}(x)$ have no common factor in $\mathbb{F}_p[x]$ (This will finish the proof, because one can then do the same thing for all $i$). We have
	$$\overline{F}(x) = \overline{F_1}(x)f_{i_1+1}(x)\cdots f_{i_l}(x)+f_1(x)\overline{K}(x).$$
	If $\overline{F}(x)$ and $f_1(x)$ had a common factor in $\mathbb{F}_p[x]$, then $\overline{F_1}(x)f_{i_1+1}\cdots f_{i_l}(x)$ and $f_1(x)$ would have a common factor in $\mathbb{F}_p[x]$. But, this would force $\overline{F_1}(x)$ and $f_1(x)$ to have a common factor in $\mathbb{F}_p[x]$, since $\text{Disc}(f)$ is relatively prime to $p$. So, if $f_1$ and $\overline{F}_1(x)$ have no common factor in $\mathbb{F}_p[x]$, we are already done. If they have a common factor, replace $\tilde{f}_1(x)$ by $\tilde{g}_1(x)=\tilde{f}_1(x)+p$, which, since $\tilde{f}_1(x)\cdots \tilde{f}_{i_1}(x)+pF_1(x)$ is a fixed polynomial in $\mathbb{Z}[x]$, will replace $F_1(x)$ by $G_1(x) = F_1(x)-\tilde{f}_2(x)\cdots \tilde{f}_{i_1}(x)$. Now $\overline{\tilde{g}}_1(x)=f_1(x)$ cannot have a common factor with $\overline{G}_1(x)$ in $\mathbb{F}_p[x]$, because $f_1(x)|\overline{F}_1(x)$ (since $f_1(x)$ and $\overline{F}_1(x)$ are assumed to have a common factor, and $f_1(x)\in \mathbb{F}_p[x]$ is irreducible), and $f_1$ is relatively prime to $f_j$ in $\mathbb{F}_p[x]$ for $j=2,\dots,i_1$ (recall that Disc$(f)$ was relatively prime to $p$). It is easy to see that we can do the same thing for each $f_i$ without affecting the fact that $f_j$ and $\overline{F}_1$ have no common factor in $\mathbb{F}_p[x]$ for $j<i$, which finishes the proof.
\end{proof}
Noting that Disc$(G_{d,0,n})$ is relatively prime to $d$ (see for instance Lemma $3$ in \cite{Buff}), now the proof of Theorem $1.9$ clearly follows from Lemma $3.4$.
\end{proof}

\begin{corollary}
Let $d$ be a prime and $m\neq 0$. Suppose that $c_0$ is a root of $G_{d,m,n}(c)$. Set $K=\mathbb{Q}(c_0)$. Then all the elements in the critical orbit $\{a_1,\dots,a_{m+n-1}\}$ of $f_{c,d}$ are square-free in $\mathcal{O}_K$.
\end{corollary}
\begin{proof}
Note that if $n\not{|}$ $i$, then $a_i$ is a unit in $\mathcal{O}_K$ by Theorem $2.1$, so there is nothing to prove. We also know from Theorem $2.1$ that $(a_n)=(a_{nk})$ in $\mathcal{O}_K$ for any $k\geq 1$, so it is enough to prove that $a_n$ is square-free in $\mathcal{O}_K$. Recall from Proposition $2.6$ that we have
\begin{equation}
(a_n) = (d,\tilde{f_1}(c_0))\cdots (d,\tilde{f_k}(c_0)),
\end{equation}
where $\tilde{f}_1,\dots,\tilde{f}_k$ are some lifts of the irreducible factors $f_1,\dots,f_k\in \mathbb{F}_d[c]$ of the reduced polynomial $\overline{G_{d,0,n}}(c)\in \mathbb{F}_d[c]$. Let $H(c)\in \mathbb{Z}[c]$ be the minimal polynomial of $c_0$. It is clear from the proof of Theorem $1.4$ that \begin{equation}
\overline{H}(c)=[f_{i_1}(c)\dots f_{i_t}(c)]^{M_{m,n}}
\end{equation}
for some $i_1,\dots, i_t\in \{1,\dots,k\}$. Then, Lemma $2.5$ implies that for any $j\in \{1,\dots k\}$, $(d,\tilde{f}_j(c_0))$ is not the unit ideal if and only if $j\in \{i_1,\dots, i_t\}$. Hence, we can rewrite (3.14) as
\begin{equation}
(a_n) = (d,\tilde{f}_{i_1}(c_0))\cdots (d,\tilde{f}_{i_t}(c_0)).
\end{equation}
But, since we know from Theorem $1.9$ that $d\not{|}\text{ }[\mathcal{O}_K:\mathbb{Z}[c_0]]$, combining (3.15) with Dedekind's Factorization Theorem will imply that $(d,\tilde{f}_{i_1}(c_0)),\dots, (d,\tilde{f}_{i_t}(c_0))$ are distinct prime ideals in $\mathcal{O}_K$, which proves that $a_n$ is square-free in $\mathcal{O}_K$, as desired.
\end{proof}
\begin{remark}
The author's interest in Corollary $3.5$ comes from the questions related to the irreducibility of iterates of polynomials. For a field $K$, we call a polynomial $f(x)\in K[x]$ \textbf{stable} if all of its iterates are irreducible over $K$. In our special case, it is known that $f_{c,d}$ is stable if the critical orbit of $f_{c,d}$ does not contain $\pm d$th power (Theorem $8$, \cite{Hamblen}). Corollary $3.5$ implies that non-unit elements in the orbit cannot be $\pm d$th power. This establishes stability in the case $n=1$, because in that case there is no unit in the orbit (by Theorem $2.1$). This was already proven in (Corollary 1.2, \cite{G}). In other words, Corollary $3.5$ can be thought of as a mild generalization of (Corollary 1.2, \cite{G}).
\end{remark}
     \subsection*{Acknowledgments}
     The author would like to thank Nigel Boston for helpful conversations related to this work. The author also thanks Rafe Jones for the helpful conversation during one of his visits to UW-Madison. It is also a pleasure to thank Rafe Jones for useful feedback on an early draft of this paper.
	
		%% Use the widest label as parameter above.
		%% Reference items can be numbered or have labels of your choice, as below.
		%% Arrange the items in the alphabetical order of names (and not in the order of labels).
		
		%% In IMPAN journals, only the title is italicized; boldface is not used.
		%% Do NOT give the issue number unless the issues are paginated separately, as in Uspekhi below.
		
		%% To ease editing, add:
		
		\baselineskip=17pt

		%%%%%%%%%%%%%

\end{document}